\documentclass[12pt]{article}
\usepackage{amsmath}
\usepackage{subfigure}% Support for small, `sub' figures and tables
\usepackage{float}
\usepackage{amsthm}
\usepackage{amssymb, setspace}
\usepackage{color,graphicx, epsfig, geometry, hyperref,fancyhdr}
\usepackage{mathrsfs}
\usepackage{bbm}
\usepackage[T1]{fontenc}
\usepackage{latexsym,amssymb,amsmath,amsfonts}
\usepackage{color}
\usepackage{soul}
\usepackage{caption}
\usepackage{dsfont}
\usepackage{multirow}
\usepackage{rotating}
\newtheorem{theorem}{Theorem}[section]

\newcommand{\N}{\mathbb{N}}

\newcommand{\R}{\mathbb{R}}
\newcommand{\C}{\mathbb{C}}

\newcommand{\grad}{\nabla}

\begin{document}
\setlength{\parskip}{1mm}
\setlength{\oddsidemargin}{0.1in}
\setlength{\evensidemargin}{0.1in}
\lhead{}
\rhead{}

\begin{flushleft}
\Large 
\noindent{\bf \Large The direct and inverse problem for sub-diffusion equations with a generalized impedance subregion}
\end{flushleft}

\vspace{0.2in}

{\bf  \large Isaac Harris}\\
\indent {\small Department of Mathematics, Purdue University, West Lafayette, IN 47907 }\\
\indent {\small Email: \texttt{harri814@purdue.edu}}

\vspace{0.2in}

%%%%%%%%%%%%%%%%%%%%%%%%%%%%%%%%%%%%%%%%%%%%
\begin{abstract}
\noindent In this paper, we consider the direct and inverse problem for time-fractional diffusion in a domain with an impenetrable subregion. Here we assume that on the boundary of the  subregion the solution satisfies a generalized impedance boundary condition. This boundary condition is given by a second order spatial differential operator imposed on the boundary. A generalized impedance boundary condition can be used to model corrosion and delimitation. The well-posedness for the direct problem is established where the Laplace Transform is used to study the  time dependent boundary value problem. The inverse impedance problem of determining the parameters from the Cauchy data is also studied provided the boundary of the subregion is known. The uniqueness of recovering the boundary parameters from the Neumann to Dirichlet mapping is proven.  
\end{abstract}

\vspace{0.1in}



%%%%%%%%%%%%%%%%%%%%%%%%%%%%%%%%%%%%%%%%%%%%%%%%%%%%%%%%%%%%%%%%%%%%%%%%
\section{Introduction}

Here we are interested in studying the direct and inverse problem for a sub-diffusive partial differential equation in a domain with an impenetrable subregion. To close the system we require that the solution has a given flux on the outer boundary and satisfies a homogeneous generalized impedance boundary condition on the interior boundary. We assume that the model is given by the fractional diffusion equation where the spatial partial differential operator is given by a symmetric elliptic operator. The temporal derivative is given by the Caputo fractional derivative denoted $\partial^{\alpha}_t$ for some fixed $\alpha \in (0,1)$. There has been a lot of interest in the study of sub-diffusive process in recent years see for e.g. \cite{frac-tutorial} and the references therein. It has even been shown in \cite{reg-w-frac} that sub-diffusive processes can be used as a regularization strategy for classical diffusive processes. In general, we have seen in the literature that the generalized impedance boundary condition models complex features such as coating and corrosion. In \cite{delamination} a generalized impedance condition is derived to asymptotically describe delimitation for the acoustic scattering problem. In \cite{gibc-eit} the factorization method is employed to solve the inverse shape problem of recovering an inclusion with a generalized impedance condition from electrostatic data and unique recovery of the impedance coefficients is proven. Recently, in \cite{heat-fm} the factorization method was studied for a heat equation to reconstruct interior cavities. The interior cavity is given by a thermal insulating region which gives a zero flux on the interior boundary. See the manuscript \cite{kirschbook} for an in-depth study of the factorization method applied to inverse scattering problems. Even though it is not considered here the question of employing the factorization method to recover the interior boundary is an interesting open problem for either the heat equation or the sud-diffusive equation.  See \cite{CK-GIBC,2nd-order-inclusion} for other examples of the inverse problem for recovering the impedance coefficients from electrostatic data. Just as in these manuscripts we are interested in the inverse impedance problem of unique recovery of the impedance coefficients. Here we will assume that we have the Cauchy data coming from the fractional diffusion equation. 

Just as in \cite{numeric-heatlt} we will use the Laplace Transform to study the well-posedness of a diffusion equation. In order to prove solvability in the time-domain, we will formally take the Laplace Transform of the time-fractional diffusion equation in question then appealing to Laplace Inversion Formula from Chapter 3 of \cite{sayas-book}. The Laplace and Fourier transforms are very useful tools for studying time-domain problems. In many manuscripts such as \cite{elastic-laplace,numeric-heatlt} the Laplace and Fourier transform are used to prove the solvability of hyperbolic and parabolic equations. This is done by reducing the time-domain to an auxiliary problem in the frequency-domain where one proves well-posedness for the auxiliary problem. In order to establish well-posedness in the time-domain, one must establish explicit bounds on the frequency variable and appeal to the inverse transform. Once in the frequency-domain, one can employ techniques used for elliptic equations.

The rest of the paper is organized as follows. In Section \ref{sect-problem-statement} we rigorously define the direct and inverse problems under consideration.  To due so, we will define the boundary value problem that will be studied as well as the appropriate assumption on the coefficients. Then in Section \ref{sect-direct-problem}, we prove well-posedness of the direct problem by studying the corresponding problem in the frequency domain given by the Laplace Transform of the time dependent problem. Section \ref{sect-inverse-problem} is dedicated to studying the inverse impedance problem of recovering the generalized impedance boundary parameters from the knowledge of the Neumann-to-Dirichlet mapping. Lastly, in the final section, we conclude by summarizing the result from the previous sections and discuss future problems under consideration.

%%%%%%%%%%%%%%%%%%%%%%%%%%%%%%%%%%%%%%%%%%%%%%%%%%%%%%%%%%%%%%%%%%%%%%%%
\section{Problem statement}\label{sect-problem-statement}
In this section, we will formulate the direct and inverse problem to be analyzed in Sections \ref{sect-direct-problem} and \ref{sect-inverse-problem}. The problems will be rigorously defined so that we may employ variational methods for solving these problems. We begin by considering the direct problem associated with the sub-diffusion equation with an impenetrable interior inclusion with a generalized impedance boundary condition. Here, we let $D \subset \R^2$ be a simply connected open set with $C^2$-boundary $\Gamma_\text{1}$ with unit outward normal $\nu$. Now let $D_0 \subset D$ be (possible multiple) connected  open set with $C^2$-boundary $\Gamma_\text{0}$, where we assume that $\text{dist}(\Gamma_\text{1} , \overline {D}_0)\geq d>0.$ This given that the annular region $D_1= D \setminus \overline{D}_0$ is a connected set with boundary $\partial D_1 = {\Gamma_\text{1}} \cup {\Gamma_\text{0}}$. See Figure \ref{pic} for example. 
\begin{figure}[ht]
\centering
\includegraphics[scale=0.23]{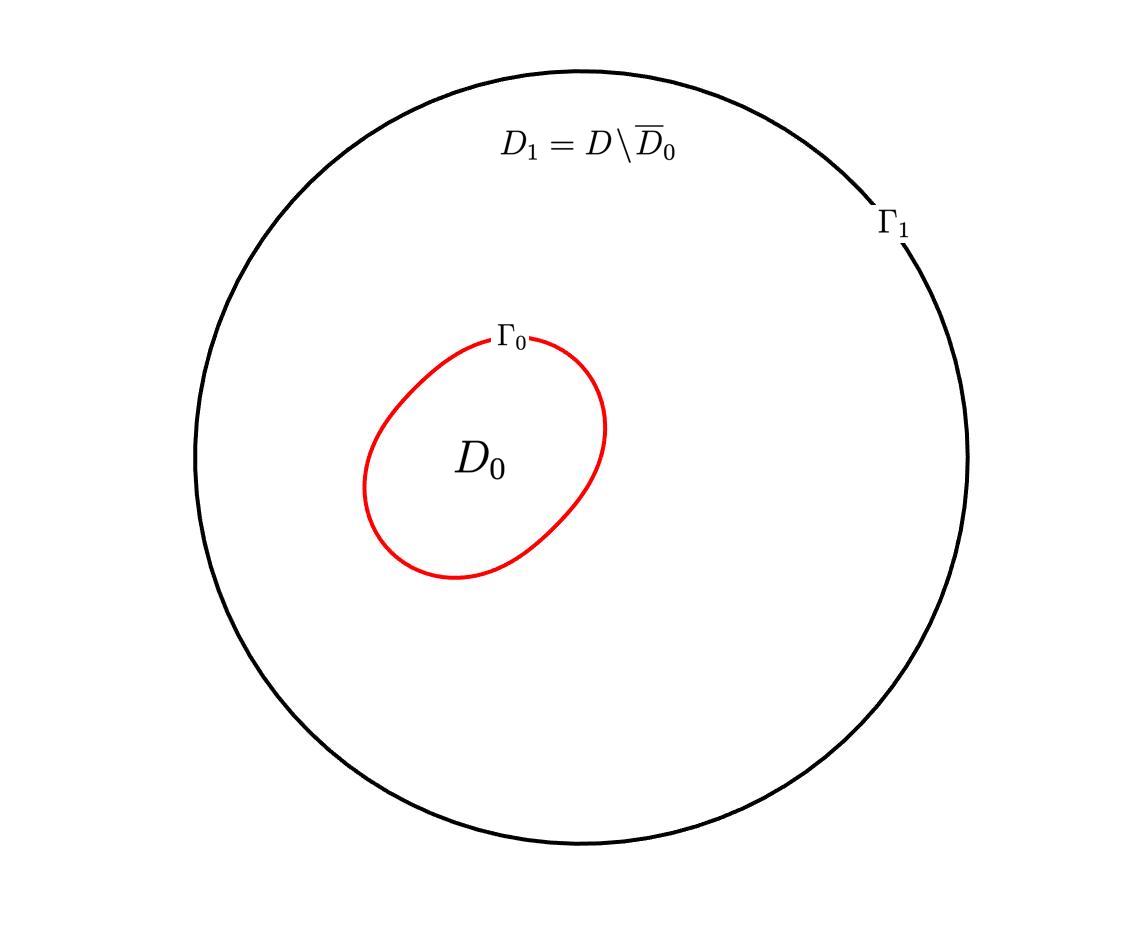}
\caption{Example of a circular domain $D$ with and elliptical subregion $D_0$.}
\label{pic}
\end{figure}

In order to study this problem, we will consider the space of tempered distribution which vanish for $t\leq 0$ (i.e. causal). Now we define $u(x,t)$ as the causal tempered distribution solution to the sub-diffusion equation with a generalized impedance boundary condition that takes values in $H^1(D_1)$ for any $t>0$. The boundary value problem under consideration is given by 
\begin{align}
\partial_t^{\alpha} u= \grad \cdot A(x) \grad u - c(x)u \quad \text{in} \quad D_1\times \R_{+}  \quad \text{with} \quad u(x,t)=0 \quad \text{for all} \,\,\, t\leq 0 \label{direct1}\\
 \partial_{\nu_{A}} u( \cdot \, ,t)  \big|_{\Gamma_\text{1}}= f(x) g(t) \quad \text{and} \quad \mathscr{B} \big[u( \cdot \, ,t)\big] \big|_{\Gamma_\text{0}}=0 \quad \text{for all} \,\,\, t> 0. \label{direct2}
\end{align} 
We will assume that the parameter $\alpha \in (0,1)$ is fixed. The fractional time derivative is assumed to be the Caputo derivative defined by
$$ \partial_t^{\alpha} u =\frac{1}{\Gamma(1-\alpha)} \int\limits_{0}^t \frac{ \partial_\tau u(\cdot \, , \tau)}{(t-\tau)^\alpha} \, \text{d} \tau$$
where $\Gamma(1-\alpha)$ is the Gamma function evaluated at $1-\alpha$. Here the boundary operator in \eqref{direct2} is defined as 
\begin{align}
\mathscr{B} \big[u\big] =  \partial_{\nu_{A}}  u -   \frac{ \text{d} }{\text{d} \sigma}  {\eta}(x)   \frac{\text{d}  }{\text{d} \sigma}  u +  {\gamma}(x) u   \label{GIBC}
\end{align}
where ${\text{d} }/{\text{d} \sigma}$ is the tangential derivative and $\sigma$ is the arc-length parameter on $\Gamma_0$. 
Here we take $\nu$ to be the unit outward normal to the domain $D_1$ and $ \nu \cdot A \grad = \partial_{\nu_{A}}  $ is the corresponding conormal derivative. Also, the generalized impedance boundary condition on the boundary $\Gamma_0$ is understood in the weak sense such that 
$$  0 = \int\limits_{\Gamma_{0} } \overline{\varphi} \partial_{\nu_{A}}u(\cdot ,t) + \eta\,   \frac{\text{d}  u(\cdot ,t)}{\text{d} \sigma}   \frac{\text{d}   \overline{\varphi} }{\text{d} \sigma} + \gamma  \, u(\cdot ,t)  \overline{\varphi} \, \text{d} \sigma \quad \text{ for all } \quad \varphi \in H^1(\Gamma_0) \,\,\,\text{ and } \,\, t>0.$$

To study the problem \eqref{direct1}--\eqref{direct2} we assume that the spatial partial differential operator is symmetric and elliptic. To this end, we let the matrix-valued coefficient  $ {A(x)} \in  C^{0,1} \left( D_1, \R^{2 \times 2} \right)$ be symmetric-positive definite such that 
$$\overline{\xi} \cdot A(x) \xi \geq A_{\text{min}} |\xi|^2 \quad \text{ for a.e } \,\, x\in D_1.$$  
The scalar coefficient $c(x) \in  L^{\infty}(D_1)$ is  such that 
$$c(x) \geq 0\quad \text{ for a.e } \,\, x\in D_1.$$  
Notice that the assumptions on the coefficients give that the differential operator defined by the right hand side of \eqref{direct1} is a symmetric elliptic partial differential operator. The flux on the boundary is given by the separated function $f(x) g(t)$ where $f \in H^{-1/2}(\Gamma_\text{1})$ and $g$ is a piecewise continuous for all $t\geq0$ of exponential order such that $g(t)=0$ for all $t\leq0$. Now, assume that the impedance parameters $\eta \in L^{\infty }(\Gamma_{0})$ and $\gamma \in L^{\infty }(\Gamma_{0})$. For analytical considerations throughout the paper, we will assume that the coefficients satisfy 
$$\eta \geq \eta_{\text{min}} >0 \quad \text{ and } \quad \gamma \geq \gamma_{\text{min}} >0 \quad \text{ for a.e } \,\, x\in \Gamma_{0}.$$
Note that in three spatial dimensions the operator $ \frac{\text{d} }{\text{d} \sigma}  {\eta}   \frac{\text{d} }{\text{d} \sigma}$ is replaced by the Laplace-Beltrami operator $\text{div}_{\Gamma_\text{0}} \big(\eta \text{ grad}_{\Gamma_\text{0}} \big)$ and the analysis in Sections \ref{sect-direct-problem} holds. The analysis in Section \ref{sect-inverse-problem} does not hold in three spatial dimensions. There is little known for the recovery on the impedance parameters in three dimensions. Also, the analysis presented in the following sections can be simply augmented for the classical diffusion process where the fractional derivative is replaced with the classical first order temporal derivative.

For completeness, we will state the result that will be used in Section \ref{sect-direct-problem} to prove well-posedness. This gives a characterization of which analytic functions with values in Banach Space is the Laplace Transform of a causal tempered distribution. To this end, let $\C_{+}= \{ s\in \C \,\, \text{ where } \, \, \text{Re}(s)>0 \}$ and $X$ be a Banach space. Assume that the mapping $\Phi : \C_{+} \longmapsto X$ is an analytic function such that 
\begin{align}
 \| \Phi(s) \|_X \leq \mathcal{C} \big( \text{Re}(s) \big) |s|^\mu \quad \text{with } \quad \mu <-1\label{lt-ref}
 \end{align}
where $\mathcal{C}:(0,\infty) \mapsto (0,\infty)$ is non-increasing with $\mathcal{C}(\sigma)=\mathcal{O}(\sigma^{-l})$ as $\sigma \to 0$ for some $l \in \N$. Then there exists a unique $X$-valued causal tempered distribution $\varphi(t)$ whose Laplace Transform is $\Phi(s)$ see \cite{sayas-book} Chapter 3 for details.

Now that we have formulated the direct problem we now define the inverse problem under consideration. Here we are interested in the inverse impedance problem of determining the boundary operator $\mathscr{B}$ (i.e. the impedance parameters) from the knowledge of the solution $u$ on the outer boundary of $\Gamma_1$. To this end, assume that the temporal function $g$ is fixed and that we have that {\it Neumann-to-Dirichlet } (NtD) mappings denoted by $\Lambda$ that maps 
$$f \longmapsto u ( \cdot \, ,t)  \big|_{\Gamma_\text{1}} \quad \text{for all} \,\,\, t> 0.$$
It is clear $\Lambda$ depends on the boundary parameters and we wish to study the injectivity of the mapping 
$(\eta , \gamma)  \longmapsto \Lambda.$ Since the temporal function $g$ is assumed to be fixed we have that the NtD operator can be viewed as  linear mapping given by 
$$\Lambda f = u ( \cdot \, ,t)  \big|_{\Gamma_\text{1}}$$
for any $f \in H^{-1/2}(\Gamma_\text{1})$. In our analysis, we will assume the knowledge fo the NtD mapping. A similar inverse impedance problems have been considered in \cite{gibc-eit} for the Electrical Impedance Tomography problem.

%%%%%%%%%%%%%%%%%%%%%%%%%%%%%%%%%%%%%%%%%%%%%%%%%%%%%%%%%%%%%%%%%%%%%%%%
\section{Analysis of the direct problem }\label{sect-direct-problem} 
To analyze the direct problem we will use the Laplace Transform. Therefore, let X be a Banach Space where we let $\text{TD} \left[ X \right]$ denote the $X$-valued causal tempered distribution with values in the Banach space $X$(see \cite{sayas-book} for details).  In order for the solution $u( \cdot \, ,t)$  of \eqref{direct1}--\eqref{direct2} to be a causal tempered distribution we will assume that $g$ is a causal (real-valued) piecewise continuous function for all $t \geq 0$ of exponential order. This gives that the boundary data $f(x) g(t)$ is a causal tempered distribution with values in $H^{-1/2}(\Gamma_\text{1})$.
We define that Laplace Transform for a causal tempered distribution $w \in \text{TD} \left[ X \right]$ as 
$$ \mathscr{L} \big\{w(t) \big\} = \int\limits_{0}^{\infty} w(t) \text{e}^{-st} \, \text{d}t \quad \text{denoted} \quad W(s)=\mathscr{L} \big\{w(t) \big\} $$
for any $s\in \C_{+}= \{ s\in \C \,\, \text{ where } \, \, \text{Re}(s)>0 \}$. By our assumptions on $g(t)$ we have that the Laplace Transform of the boundary data exists and is given by $f(x)G(s)$ where $\mathscr{L} \big\{g(t) \big\} =G(s)$. We will further assume that there is a constant independent of $s \in  \C_{+}$ where the Laplace Transform for $g$ satisfies 
\begin{align}
|G(s)| \leq  \frac{C}{|s|^p} \quad \text{ for some } \quad p>1 \quad \text{ for all } \quad s \in  \C_{+}. \label{g-decay}
\end{align}

Now we consider the function space for the solution to the direct problem. Due to the generalized impedance condition \eqref{GIBC} we consider the solution as a causal tempered distribution that has values in $H^1( D_1 , \Gamma_\text{0} )$.
Therefore, we wish to show the existence and uniqueness of the solution $u \in \text{TD} \left[ H^1( D_1 , \Gamma_\text{0} ) \right]$ that is the solution to \eqref{direct1}$-$\eqref{direct2} for given boundary data $f(x) g(t) \in \text{TD} \left[ H^{-1/2}(\Gamma_\text{1}) \right]$. We now define the space where for which the we attempt to find the solution as 
$$ H^1( D_1 , \Gamma_\text{0} ) =  \Big\{  \varphi \in {H}^1 (D_1) \quad \text{such that} \quad \varphi \big|_{\Gamma_\text{0}} \in H^1 (\Gamma_\text{0})  \Big\}$$
with that associated norm/inner-product
$$ \| \varphi  \|^2_{H^1( D_1 , \Gamma_\text{0} ) }= \| \varphi  \|^2_{H^1( D_1) }  +\left\| {\varphi}  \right\|^2_{H^1(\Gamma_\text{0})} .$$
It is clear that $H^1( D_1 , \Gamma_\text{0} )$ is a Hilbert Space with the graph norm defined above.  
Here the Sobolev Spaces on the boundary are defined by the dual pairing between $H^{p}(\Gamma_j)$ and $H^{-p}(\Gamma_j)$ (for $p \geq 0$) with $L^2(\Gamma_j)$ as the pivot space where $\Gamma_j$ for $j=0,1$ are the closed curves defined in the previous section.
The definition of the aforementioned Sobolev Spaces can be found in \cite{evans,salsa}.

In order to prove the well-posedness of \eqref{direct1}$-$\eqref{direct2} with respect to any given spatial boundary data $f \in H^{-1/2}(\Gamma_\text{1})$ and fixed causal temporal data $g$ satisfying \eqref{g-decay}, we use the Laplace Transform. We formally take the Laplace Transform of equation \eqref{direct1}$-$\eqref{direct2} and by appealing to the fact that the solution $u$ is causal to obtain 
 \begin{align}
- \grad \cdot A(x) \grad U + (c(x) + s^{\alpha} )U =0 \quad \text{in} \quad D_1 \quad   \text{for all} \quad s \in  \C_{+} \label{lt-direct1}\\
 \partial_{\nu_{A}} U( \cdot \, ; s)  \big|_{\Gamma_\text{1}}= f(x) G(s) \quad \text{and} \quad \mathscr{B} \big[U( \cdot \, ;s)\big] \big|_{\Gamma_\text{0}}=0 \quad \text{for all} \,\,\, s \in  \C_{+}. \label{lt-direct2}
\end{align} 
Here, $U( \cdot \, ; s) $ denotes that Laplace Transform of $u(\cdot \, ,t)$. We have used the fact that 
$$ \mathscr{L} \big\{\partial_t^{\alpha} u( \cdot , t) \big\} = s^{\alpha} U( \cdot \, ; s) $$
by appealing to the definition of the fractional time derivative and the Convolution Theorem for Laplace Transforms. We can consider \eqref{lt-direct1}$-$\eqref{lt-direct2} as the frequency-domain boundary value problem associated with  \eqref{direct1}$-$\eqref{direct2}. Using the  Laplace (or Fourier) Transform to study time-domain problems is commonly done for hyperbolic problems (see for e.g. \cite{elastic-laplace,TDLSM,wavelsm}). To prove the well-posedness of \eqref{direct1}$-$\eqref{direct2} we will need to show that \eqref{lt-direct1}$-$\eqref{lt-direct2} is well-posed and then by appeal to the Laplace Inversion Theorem (\cite{sayas-book} Chapter 3). This means we need to prove that $U( \cdot \, ; s)$ satisfies the estimate \eqref{lt-ref}. To this end, we will employ a variational technique for proving the well-posedness of \eqref{lt-direct1}$-$\eqref{lt-direct2} where we must establish estimates where the dependence on the frequency variable $s\in  \C_{+}$ is explicit. 

We have that, for any given $V \in H^1( D_1 , \Gamma_\text{0} )$ the equivalent variational formulation of \eqref{lt-direct1}$-$\eqref{lt-direct2} is obtained by appealing to Green's 1st Theorem and is given by 
\begin{align}
a_s(U , V) + b(U , V) = \ell_s (V). \label{lt-varform}
\end{align} 
Here the sesquilinear forms $a_s( \cdot \, ,  \cdot )$ and  $b( \cdot \, ,  \cdot ) : H^1( D_1 , \Gamma_\text{0} )^2  \longmapsto \C$ are defined by 
\begin{align}
a_s(U , V) &= \int\limits_{ D_1 }  A(x) \grad U \cdot \grad \overline{V} + (c(x) + s^{\alpha} )U \overline{V}  \, \text{d} x,  \label{a-varform} \\
b(U , V) &= \int\limits_{\Gamma_{0} } \eta\,   \frac{\text{d}  U}{\text{d} \sigma}   \frac{\text{d}  \overline{V} }{\text{d} \sigma} +\gamma  \, U  \overline{V}  \, \text{d} \sigma \label{b-varform}
\end{align} 
and the conjugate linear functional $\ell_s (\cdot ) : H^1( D_1 , \Gamma_\text{0} )  \longmapsto \C$ is defined as 
\begin{align}
\ell_s (V)= G(s) \int\limits_{\Gamma_{1} } f  \, \overline{V}  \, \text{d} \sigma. \label{l-varform}
\end{align} 
It is clear that the sesquilinear forms are continuous for any given $s\in  \C_{+}$ by appealing to the boundedness of the coefficients and the Cauchy-Schwartz inequality. In order to prove the well-posedness we will use the Lax-Milgram Lemma (see \cite{salsa} Theorem 6.5) where the coercivity constant will depend on $s$. Then, in-order to prove that the solution $U( \cdot \, ; s)  \in H^1( D_1 , \Gamma_\text{0} )$ to \eqref{lt-varform} (and therefore \eqref{lt-direct1}$-$\eqref{lt-direct2}) is the Laplace Transform of a tempered distribution $u \in \text{TD} \left[ H^1( D_1 , \Gamma_\text{0} ) \right]$ that solves \eqref{direct1}$-$\eqref{direct2} we prove that the reciprocal of the coercivity constant satisfies the assumption of the Laplace Inversion Formula given by equation (3.2) in \cite{sayas-book}.

\begin{theorem}
The sesquilinear form $a_s( \cdot \, ,  \cdot )$ defined by \eqref{a-varform} satisfies the estimate 
$$| a_s( U,U ) | \geq C \cos(\alpha \pi/2) \text{min}\big(1,Re(s)^\alpha \big) \| U\|_{H^1(D_1)}^2 $$
where $C>0$ is a constant depending only on the coefficient matrix. 
\end{theorem}
\begin{proof}
To prove the claim, notice that  
\begin{align*}
\big| a_s( U,U ) \big| & = \big| \text{e}^{- \text{i} \alpha \text{Arg}(s)} a_s( U,U ) \big|\\
			        & \geq \big| \text{Re} \left(\text{e}^{- \text{i} \alpha \text{Arg}(s)} a_s( U,U )  \right)\big|.
\end{align*}
Here $\text{Arg}(s)$ denoted the argument of the complex number (i.e. the angular variable when represented in polar coordinates) such that $s=|s|\text{e}^{\text{i} \alpha \text{Arg}(s)}$. Recall, that fo any $s \in \C_{+}$ which gives that $|\alpha \text{Arg}(s)| \leq \alpha\pi/2$ and therefore 
$$1\geq \cos\left( \alpha \text{Arg}(s) \right) \geq \cos(\alpha\pi/2)>0 \quad \text{for all} \quad \alpha \in (0,1).$$ 
Now using the fact the $c(x) \geq 0$ we can then estimate 
\begin{align*}
\text{Re} \left(\text{e}^{- \text{i} \alpha \text{Arg}(s)} a_s( U,U )  \right) & \geq \cos\left( \alpha \text{Arg}(s) \right) \int\limits_{ D_1 }  A(x) \grad U \cdot \grad \overline{U} \, \text{d} x  +  |s|^{\alpha} \int\limits_{ D_1 } |U|^2 \, \text{d} x \\
			        & \geq \cos(\alpha\pi/2)  \Big[A_{\text{min}} \| \grad U\|^2_{L^2(D_1)} + \text{Re}(s)^\alpha \| U\|^2_{L^2(D_1)} \Big]\\
			        & \geq \cos(\alpha\pi/2) \text{min}(1,A_{\text{min}})  \text{min}\left(1, \text{Re}(s)^\alpha \right)\| U\|_{H^1(D_1)}^2.
\end{align*}
This proves the claim. 
\end{proof}

This gives us an explicit $s$ dependent coercivity estimate in $H^1(D_1)$ for $a( \cdot \, ,  \cdot )$. We now prove a coercivity estimate in $H^1(\Gamma_0)$ for the sesquilinear form $b( \cdot \, ,  \cdot )$ which would imply that the sum of the sesquilinear forms is coercive in $H^1(D_1 , \Gamma_0)$.

\begin{theorem}
The sesquilinear form $b( \cdot \, ,  \cdot )$ defined by \eqref{b-varform} satisfies the estimate 
$$| b( U,U ) | \geq C \cos(\alpha \pi/2) \| U\|_{H^1(\Gamma_0)}^2 $$
where $C>0$ is a constant depending only on the impedance parameters. 
\end{theorem}
\begin{proof}
Similarly to prove the lower bound we consider 
\begin{align*}
\big| b( U,U ) \big| \geq \big| \text{Re} \left(\text{e}^{- \text{i} \alpha \text{Arg}(s)} b( U,U )  \right)\big|
\end{align*}
where again $\text{Arg}(s)$ is the argument of the complex number $s$. We still have that 
$$1\geq \cos\left( \alpha \text{Arg}(s) \right) \geq \cos(\alpha\pi/2)>0 \quad \text{for all} \quad \alpha \in (0,1).$$ 
Now, using the lower bounds on the impedance parameters 
$$\eta \geq \eta_{\text{min}} >0 \quad \text{ and } \quad \gamma \geq \gamma_{\text{min}} >0 \quad \text{ for a.e } \,\, x\in \Gamma_{0}$$
we have that 
\begin{align*}
\text{Re} \left(\text{e}^{- \text{i} \alpha \text{Arg}(s)} b( U,U )  \right) & \geq \cos\left( \alpha \text{Arg}(s) \right) \left[ \int\limits_{\Gamma_{0} } \eta\,   \left| \frac{\text{d}  U}{\text{d} \sigma}  \right|^2 +\gamma  \, |U|^2   \, \text{d} \sigma \right] \\
			        & \geq \cos(\alpha\pi/2)  \left[\eta_{\text{min}} \left\| \frac{\text{d}  U}{\text{d} \sigma}   \right\|^2_{L^2(\Gamma_{0})} + \gamma_{\text{min}} \left\| U   \right\|^2_{L^2(\Gamma_{0})} \right]\\
			        & \geq \cos(\alpha\pi/2) \text{min}(\gamma_{\text{min}},\eta_{\text{min}}) \| U\|_{H^1(\Gamma_{0})}^2.
\end{align*}
This proves the claim. 
\end{proof}

Notice that the Lax-Milgram Lemma implies that the sesquilinear form given by $a_s( \cdot \, ,  \cdot )+b( \cdot \, ,  \cdot )$ defined by \eqref{a-varform}--\eqref{b-varform}  can be represented by an invertible operator $\mathbb{T}(s)$ that maps  $H^1( D_1 , \Gamma_\text{0} )$ into itself such that 
$$ a_s( U ,  V )+b( U  ,  V ) = \big( \mathbb{T}(s) U , V \big)_{H^1( D_1 , \Gamma_\text{0} )} \quad \text{for all} \quad U,V \in H^1( D_1 , \Gamma_\text{0} ).$$
Since the sesquilinear form $a_s( \cdot \, ,  \cdot )$ is analytic for $s \in \C_{+}$ we have that $\mathbb{T}(s)$ depends analytically on  $s \in \C_{+}$. Now provided that $\mathbb{T}(s)$ is invertible then it's inverse would also depends analytically on $s \in \C_{+}$.

We will now derive a norm estimate for the inverse of $\mathbb{T}(s)$ for any $s \in \C_{+}$ where the dependence on the frequency variable is made explicit. To this end, he lower bounds given in the above results imply that 
$$ \Big| \big( \mathbb{T}(s) U , U \big)_{H^1( D_1 , \Gamma_\text{0} )} \Big| \geq C  \cos(\alpha\pi/2)  \text{min}\left(1, \text{Re}(s)^\alpha \right)\| U\|_{H^1(D_1 , \Gamma_\text{0} )}^2$$
where the constant $C$ is independent of $s \in \C_{+}$. Notice that we have used that 
$$\text{Re} \left(\text{e}^{- \text{i} \alpha \text{Arg}(s)} b( U,U )  \right) \geq C \cos(\alpha \pi/2) \| U\|_{H^1(\Gamma_0)}^2 \geq  C \cos(\alpha \pi/2) \text{min}\big(1,\text{Re}(s)^\alpha \big) \| U\|_{H^1(\Gamma_0)}^2.$$
From the coercivity estimate we have that 
$$ \left\| \mathbb{T}^{-1}(s) \right\|_{\mathcal{B}\big( H^1( D_1 , \Gamma_\text{0}) \big) }\leq \frac{C \sec(\alpha\pi/2) }{\text{min}\left(1, \text{Re}(s)^\alpha \right)} $$
(see \cite{salsa} Theorem 6.5) in the operator norm where $\mathcal{B}\big( H^1( D_1 , \Gamma_\text{0}) \big)$  is the space of bounded linear transformations form $H^1( D_1 , \Gamma_\text{0})$ into itself. Therefore, in the inversion theorem we can conclude that $\mathcal{C} \big( \text{Re}(s) \big)$ from \eqref{lt-ref} is given by 
$$\mathcal{C} \big( \text{Re}(s) \big) =  \frac{C \sec(\alpha\pi/2) }{\text{min}\left(1, \text{Re}(s)^\alpha \right)}$$
Now we derive a norm estimate for the conjugate linear functional $\ell_s (\cdot )$.

\begin{theorem}
The conjugate linear functional $\ell_s (\cdot )$ defined by \eqref{l-varform} satisfies the estimate 
$$| \ell_s( V ) | \leq  C |G(s)| \, \|f\|_{H^{-1/2}(\Gamma_1)} \|V\|_{H^1(D_1,  \Gamma_0)}$$
where $C>0$ is a constant depending only on the domain. 
\end{theorem}
\begin{proof}
This is a consequence of the Duality between $H^{\pm1/2}$  with $L^2$ as the pivot space and the Trace Theorem (see for e.g. \cite{evans}) which gives that 
\begin{align*}
\big| \ell_s (V) \big| &\leq  |G(s)| \,  \|f\|_{H^{-1/2}(\Gamma_1)} \|V\|_{H^{1/2}(\Gamma_0)}\\
			     & \leq  C |G(s)| \, \|f\|_{H^{-1/2}(\Gamma_1)} \|V\|_{H^1(D_1 , \Gamma_0)}.
\end{align*}
Here $C$ is the constant from the Trace Theorem, proving the claim. 
\end{proof}

By appealing to the Reisz Representation Theorem we can conclude that the variational problem \eqref{lt-varform} is equivalent to 
$$ \mathbb{T}(s) U = L_s \quad \text{ where } \quad \ell_s( V ) = (L_s , V)_{H^1( D_1 , \Gamma_\text{0} )} \,\, \text{for all} \,\, V \in H^1( D_1 , \Gamma_\text{0} ).$$
Where we have that 
$$\| L_s \|_{H^1( D_1 , \Gamma_\text{0} )}   \leq  C |G(s)| \, \|f\|_{H^{-1/2}(\Gamma_0)}.$$ Provided that $G(s)$ depends analytically on $s \in \C_{+}$ we can conclude that $L_s$ depends analytically on $s \in \C_{+}$. This imples that $U=U(\cdot , s) \in H^1( D_1 , \Gamma_\text{0} )$ is given by $U(\cdot , s) = \mathbb{T}(s)^{-1} L_s$ and is therefore analytic with respect to $s \in \C_{+}$. By appealing to the estimate of the norm for the inverse of $\mathbb{T}(s)$ we have that the solution $U$ to the variational problem \eqref{lt-varform} satisfies the norm estimate
\begin{align}
\| U(\cdot , s)   \|_{H^1( D_1 , \Gamma_\text{0} )} \leq \frac{C \sec(\alpha\pi/2) }{\text{min}\left(1, \text{Re}(s)^\alpha \right)} \, |G(s)| \, \|f\|_{H^{-1/2}(\Gamma_1)} \label{varform-wellposed}
\end{align}
where the constant $C>0$ is independent of $s \in \C_{+}$. From the above analysis we have that there is a unique solution to \eqref{lt-direct1}--\eqref{lt-direct2} satisfying the stability estimate \eqref{varform-wellposed}. We recall that \eqref{lt-direct1}--\eqref{lt-direct2} was obtained by taking the Laplace Transform of the time dependent equations \eqref{direct1}--\eqref{direct2}. In order to prove the well-posedness of  \eqref{direct1}--\eqref{direct2} we still need to show that $U(\cdot , s)$ is the Laplace Transform of some casual tempered distribution $u( \cdot \, , t)$ that takes values in $H^1( D_1 , \Gamma_\text{0} )$. To do so, we will appeal to the Laplace Inversion Theorem which can be applied since we have assumed that the Laplace transform for $g(t)$ satisfies \eqref{g-decay}.  Applying equation 3.3 from \cite{sayas-book} gives the following result.

\begin{theorem}\label{well-posed}
Assume that $f\in H^{-1/2}(\Gamma_0)$ and the Laplace Transform of $g(t)$ given by $G(s)$ depends analytically on $s \in \C_{+}$ satisfying \eqref{g-decay}. Then we have that there is a unique solution $u \in \text{TD} \left[ H^1( D_1 , \Gamma_\text{0} ) \right]$ to \eqref{direct1}--\eqref{direct2}. Moreover, we have the estimate 
$$ \| u(\cdot \, , t) \|_{H^1( D_1 , \Gamma_\text{0} )} \leq C t^{\alpha +|1-p|}  \|f\|_{H^{-1/2}(\Gamma_1)} \quad \textrm{ for all } \,\, f\in H^{-1/2}(\Gamma_1)$$
when $t\geq 1$ where the constant $C>0$ is independent of $t$. 
\end{theorem} 

Notice that Theorem \ref{well-posed} gives that there is a solution $u( \cdot \, , t)$ to \eqref{direct1}--\eqref{direct2} that has at most polynomial growth in $t$. The proof of Theorem \ref{well-posed} is a direct consequence of the previous analysis in this section along with the strong inversion formula for the Laplace Transform. The polynomial growth will play a role in an estimate in the proceeding section to study the inverse problem.

%%%%%%%%%%%%%%%%%%%%%%%%%%%%%%%%%%%%%%%%%%%%%%%%%%%%%%%%%%%%%%%%%%%%%%%%
\section{Analysis of the inverse problem }\label{sect-inverse-problem}  
In this section, we consider the inverse impedance problem of recovering the impedance parameters $\eta$ and $\gamma$ from the Cauchy data. These types of inverse problems have applications where one needs to infer about the interior structure of a medium from boundary measurements. These problems are frequently found in engineering applications of non-destructive testing. The mathematical questions are  uniqueness, existence, and continuity with respect to the given measurements as well as developing numerical inversion algorithms. These questions have been studied for the elliptic problem coming from Electrical Impedance Tomography in \cite{CK-GIBC,2nd-order-inclusion,gibc-eit} were uniqueness results are given as well as numerical methods for recovering the impedance parameters. Note that the generalized impedance condition given in \eqref{GIBC} depends on the material parameters $\eta$ and $\gamma$ linearly. Therefore, one hopes to derive a direct algorithm for recovering the coefficients. This is useful since it would not require initial estimate on the material parameters. In \cite{gibc-eit} this is done in the case of Electrical Impedance Tomography as well as developed a factorization method for recovering the interior boundary. Whereas in \cite{CK-GIBC} a system of non-linear boundary integral equations is used to recover the impedance parameters and interior boundary $\Gamma_0$. Here we will only focus on the question of uniquely determining the impedance parameters on the interior boundary from measurement on the exterior boundary. 

To begin, we assume that the temporal part of the flux $g(t)$ is a causal tempered distribution that is again {\it fixed} such that it's Laplace Transform $G(s)$ is well-defined and depends analytically on $s \in \C_{+}$ satisfying \eqref{g-decay}. Therefore, by Theorem \ref{well-posed} we have that there is a unique solution  $u$ to \eqref{direct1}--\eqref{direct2} that is a  causal tempered distribution that takes values in $H^1( D_1 , \Gamma_\text{0} )$ for all $t>0$. Then we consider the {\it Neumann-to-Dirichlet } (NtD) mappings denoted by $\Lambda$ that maps 
$$H^{-1/2}(\Gamma_1) \to \text{TD} \left[H^{1/2}(\Gamma_1) \right]$$
such that 
$$f \longmapsto u ( \cdot \, ,t)  \big|_{\Gamma_\text{1}} \quad \text{for all} \,\,\, t> 0.$$
By appealing to Theorem \ref{well-posed} and the Trace Theorem we have that the NtD operator is a well defined linear operator. The main idea in this section is to extend the theory developed in \cite{gibc-eit} for the Electrical Impedance Tomography problem for our inverse problem by appealing to the Laplace Transform. This employs variational techniques to prove the uniqueness of the coefficients from the knowledge of the NtD mapping. Since variational techniques are used less regularity is needed in the analysis than in \cite{CK-GIBC} but one requires the knowledge of the full NtD mapping. We will assume that the NtD mapping is known for any $f\in H^{-1/2}(\Gamma_1)$ and for all $t>0$ denoted
$$ \Lambda=\Lambda(\eta , \gamma) \quad \text{ with } \quad \Lambda f =  u ( \cdot \, ,t)  \big|_{\Gamma_\text{1}} \quad \text{for all} \,\,\, t> 0.$$
Since the NtD mapping is known for all $t>0$ we can consider the Laplace Transform of the NtD mapping  
$$ \mathscr{L} \big\{ \Lambda f \big\} = \int\limits_{0}^{\infty} u ( \cdot \, ,t)  \big|_{\Gamma_\text{1}} \text{e}^{-st} \, \text{d}t$$ 
that maps 
$$ f \longmapsto U(\cdot \, , s) \big|_{\Gamma_\text{1}} \quad \text{for any} \,\,\,s \in \C_{+}$$
where $U$ is the solution to \eqref{lt-direct1}$-$\eqref{lt-direct2}. Since $U$ solves an elliptic problem it is easier to study the uniqueness in the frequency-domain which would imply uniqueness in the time-domain by the inversion formula.  Before we can prove the main uniqueness result we first prove an auxiliary density result.

\begin{theorem}\label{dense-set}
Define the set 
$$\mathcal{U}=\Big\{ U \big|_{\Gamma_0} \, :  \,  U \in {H}^1( D_1,\Gamma_0)  \, \text{  solving  }\eqref{lt-direct1}-\eqref{lt-direct2} \,  \text{ for all }  \, f \in  H^{-1/2}(\Gamma_1) \Big\} \subset H^1(\Gamma_0).$$
Then $\mathcal{U}$ is a dense subspace of $L^2(\Gamma_0)$ for any $s \in \R_{+}$ such that $G(s) \neq 0$. 
\end{theorem}
\begin{proof}
It is clear that the mapping $f \mapsto U(\cdot , s) \big|_{\Gamma_0}$ is linear since it is the composition of the solution operator for \eqref{lt-direct1}$-$\eqref{lt-direct2} and the Trace operator. This implies the $\mathcal{U}$ defines a linear subspace of $L^2(\Gamma_0)$. Now to prove the claim we will show that the set $\mathcal{U}^{\perp} =\{0\}$. To this end, we let $\varphi \in \mathcal{U}^{\perp}$ and let $V \in {H}_0^1 (D_1,\Gamma_\text{1})$ be the solution to the dual problem 
 \begin{align*}
- \grad \cdot A(x) \grad V + (c(x) + \overline{s}^{\alpha} )V =0 \quad \text{in} \quad D_1 \quad   \text{for all} \quad s \in  \C_{+}\\
 \partial_{\nu_{A}} V  \big|_{\Gamma_\text{1}}=  0  \quad \text{and} \quad \mathscr{B} \big[V\big] \big|_{\Gamma_\text{0}}=\varphi \quad \text{for all} \,\,\, s \in  \C_{+}. 
\end{align*} 
It is clear that there is a unique solution $V \in {H}_0^1 (D_1,\Gamma_\text{0})$ to the dual problem above by appealing to similar arguments as in Section \ref{sect-direct-problem}. Now let $s \in \R_{+}$ such that $G(s) \neq 0$. Therefore, we obtain that 
\begin{align*}
0& =  \int\limits_{\Gamma_0} U \,  {\varphi} \, \text{d}s =  \int\limits_{\Gamma_0} U\, \mathscr{B} \big[V\big] \, \text{d}s  \\
&=\int\limits_{\Gamma_0} U \partial_{\nu_A} V - V \partial_{\nu_A} U \, \text{d}s  \\
&= - \int\limits_{\Gamma_1} U \partial_{\nu_A} V - V \partial_{\nu_A} U \, \text{d}s \\
&= G(s) \int\limits_{\Gamma_1} f \, V \, \text{d}s  \quad \text{for all } \quad f \in H^{-1/2} (\Gamma_\text{1})
\end{align*}
where we have used Green's 2nd Theorem.  
Due to the duality of $H^{\pm1/2}$ the Han-Banach Theorem implies that $V=0$ on $\Gamma_1$. Since $V$ has zero Cauchy data on $\Gamma_1$ we can conclude that $V=0$ in $D_1$. The generalized impedance boundary condition implies that $\varphi = 0$, proving the claim. 
\end{proof}

In order to prove the uniqueness result, we will require that the impedance parameters $(\eta , \gamma) \in C(\Gamma_0) \times L^{\infty} (\Gamma_0)$. Even though less regularity is needed to prove the well-posedness of the problem we will see that the increased regularity is needed for the proof of the uniqueness result presented in this section. This is not uncommon that the well-posedness can be established for weaker assumptions on the coefficients. The extra regularity for $\eta$ is expected since it turns up in the second order differential operator on the boundary. This is standard in the analysis of PDEs just as in standard elliptic regularity results \cite{evans}. 

\begin{theorem} \label{unique} 
Let $\Lambda$ be the Neumann-to-Dirichlet operator for \eqref{direct1}$-$\eqref{direct2} such that 
$$f \longmapsto u ( \cdot \, ,t)  \big|_{\Gamma_\text{1}} \quad \text{for all} \,\,\, t> 0.$$
Then the mapping $(\eta , \gamma) \mapsto \Lambda(\eta , \gamma)$ is injective provided that $(\eta , \gamma) \in C(\Gamma_0) \times L^{\infty} (\Gamma_0)$.   
\end{theorem}

\begin{proof}
In order to prove the claim, we proceed by way of contradiction. So assume that there are two sets of impedance parameters denoted $(\eta_j , \gamma_j ) \in C(\Gamma_0) \times L^{\infty} (\Gamma_0)$ that produce the same NtD data for all $t>0$. Then we have that the corresponding NtD mappings 
 $$\Lambda_{j}=\Lambda(\eta_j , \gamma_j) \quad \text{for } \quad  j=1,2$$
 coincide for all $f \in H^{-1/2} (\Gamma_\text{1})$. Now defined the corresponding solutions to \eqref{direct1}$-$\eqref{direct2} by $u^{(j)}$ and it's Laplace Transform by $U^{(j)}$ that is the solution to \eqref{lt-direct1}$-$\eqref{lt-direct2}. Since the Cauchy data for $u^{(j)}$ on $\Gamma_\text{1}$ coincides for for all $t>0$ we have that $U^{(1)}=U^{(2)}$ in $D_1$ for all $s\in\C_{+}$ and for any $f \in H^{-1/2} (\Gamma_\text{1})$. We will assume that $s \in \R_{+}$ so that the Laplace Transforms of the solution are real-valued. Now denote $U= U^{(1)}=U^{(2)}$ which satisfies the  generalized impedance conditions 
$$ \partial_{\nu} U - \frac{ \text{d} }{\text{d} \sigma} {\eta_1} \frac{\text{d} }{\text{d} \sigma} U+ \gamma_1 U =\partial_{\nu} U -   \frac{ \text{d} }{\text{d} \sigma} {\eta_2} \frac{\text{d} }{\text{d} \sigma} U  + \gamma_2 U  =0 \quad \text{ on } \, \, \Gamma_0.$$
By subtracting the equations we obtain 
$$0= - \frac{ \text{d} }{\text{d} \sigma} {(\eta_1 -\eta_2)} \frac{\text{d} }{\text{d} \sigma} U+ (\gamma_1-\gamma_2) {U}  \quad \text{ on } \, \, \Gamma_0$$
and integrating over $\Gamma_0$ gives that  
$$ 0 = \int\limits_{\Gamma_0}  - \frac{ \text{d} }{\text{d} \sigma} {(\eta_1 -\eta_2)} \frac{\text{d} }{\text{d} \sigma} U+ (\gamma_1-\gamma_2) U \, \text{d} \sigma =  \int\limits_{\Gamma_0} (\gamma_1-\gamma_2)U \, \text{d} \sigma $$
where the equality comes from integration by parts with the arc length variable $\sigma$. Since the above equality holds for all  $f \in H^{-1/2} (\Gamma_\text{1})$ appealing Theorem \ref{dense-set} we can conclude that $\gamma_1 = \gamma_2$ a.e. on $\Gamma_0$.

Now assume that $f \in L^2(\Gamma_1) \subset H^{-1/2} (\Gamma_\text{1})$ is real-valued, then by the similar analysis as in Section 2 of \cite{2nd-order-inclusion} we can conclude that $U \in H^{3/2}(D_1)$ which implies that $\partial_\nu U \in L^2(\Gamma_0)$. Then the generalized impedance boundary condition implies that 
$$\eta_1 \frac{\text{d} U}{\text{d} \sigma}   \in H^1(\Gamma_0) \quad \text{ for all  } \quad f \in L^2(\Gamma_1)$$   
which implies that $U \in C^1(\Gamma_0)$ since $H^1(\Gamma_0) \subset C(\Gamma_0)$ and $\eta_1 \in C(\Gamma_0)$ with $\eta_1$ strictly positive. Since $\gamma_1 = \gamma_2$ subtracting the generalized impedance conditions gives  
\begin{align*}
\frac{ \text{d} }{\text{d} \sigma} {(\eta_1 -\eta_2)} \frac{\text{d} }{\text{d} \sigma} U  \, =0  \quad \text{ for all } \, \, \, f \in L^2(\Gamma_1). 
\end{align*}
Whence 
$$ {(\eta_1 -\eta_2)} \frac{\text{d} U }{\text{d} \sigma}  =C \quad \text{ for all } \, \, \, f \in L^2(\Gamma_1)$$
where $C$ is some constant. Now define $x(\sigma): [0,\ell] \mapsto \R^2$ as an $\ell$-periodic $C^2$ representation of the closed curve $\Gamma_0$ where $\ell$ is the length of the curve. Then we identify the space $H^1(\Gamma_0)$ with the auxiliary space $H^1_{\text{per}}[0,\ell]$ of $\ell$-periodic functions. It is clear that due to the periodic condition that $U \big( x(0)\big) = U \big( x(\ell)\big)$ for all real-valued $f  \in L^2(\Gamma_1)$. Rolle's Theorem gives that the tangential derivative for $U$ is zero for at least one point on the curve which gives that 
$$ {(\eta_1 -\eta_2)} \frac{\text{d} U}{\text{d} \sigma}  =0 \quad \text{ for all real-valued } \, \, \, f \in L^2(\Gamma_1).$$ 
Now to prove that $\eta_1=\eta_2$ we proceed by contradiction and assume that there is some $x^* \in \Gamma_0$ where $(\eta_1- \eta_2)(x^*) >0$. Due to the continuity there exist $\delta >0$ such that $(\eta_1- \eta_2) >0$ for all $x \in \Gamma_0^\delta = \Gamma_0 \cap B(x^*,\delta)$. We can conclude that 
\begin{align} \label{zero-derivative}
\frac{\text{d} U }{\text{d} \sigma}  =0  \quad \text{ on } \Gamma_0^\delta \quad \text{ for all real-valued } \, \, \, f \in L^2(\Gamma_1).
\end{align}
Now for any $f_1$ and $f_2$ linearly independent real-valued $L^2(\Gamma_1)$ functions, we have that the corresponding $U_{f_1}$ and $U_{ f_2}$ are linearly independent (see Theorem 2.2 in \cite{2nd-order-inclusion}). Therefore, we can conclude that the Wronskian given by 
$$\big(U_{f_1} , U_{ f_2} \big) \longmapsto U_{f_1} \frac{ \text{d} }{\text{d} \sigma}U_{ f_2}- U_{ f_2} \frac{ \text{d} }{\text{d} \sigma}U_{f_1}$$
can not be identically zero on any open subset of $\Gamma_0$. By \eqref{zero-derivative} we have that Wronskian is identically zero on $\Gamma_0^\delta$, which contradicts the linear independence of $f_1$ and $f_2$ proving the claim.  
\end{proof}

Notice that from the proof of Theorem \ref{unique} we have that Cauchy data for $f$ and $U(\cdot \, , s)$ on ${\Gamma_\text{1}}$ uniquely determines the impedance parameters. Assuming that the NtD, as well as $\Gamma_0$, is known this implies that we can use a data completion algorithm to recover $U_f (\cdot \, , s) $ and $\partial_{\nu_A} U_f (\cdot \, , s)$ on the inner boundary $\Gamma_0$. Recently, in \cite{data-completion} a stable data completion algorithm was derived using boundary integral equations for the Helmholtz equation. Provided that $A=I$ and $c=0$ the numerical method for recovering the interior Cauchy data in \cite{data-completion} can be employed for a given $s \in \R_{+}$. Once $U_f (\cdot \, , s) $ and $\partial_{\nu} U_f (\cdot \, , s)$ are known on $\Gamma_0$  we can employ the reconstruction algorithm in Section 4 of \cite{gibc-eit} to recover the impedance parameters. This method constructs a linear system of equations to recover the impedance parameters. This gives a direct method for recovering the parameters where one does not need a prior estimates for $\eta$ and $\gamma$. To do this, we need the compute the Laplace Transform of the data. This would require infinite temporal measurements on ${\Gamma_\text{1}}$ which is not physically feasible. Therefore, we will show that one can take partial temporal measurements on the outer boundary ${\Gamma_\text{1}}$ to approximate the Laplace Transform of the NtD mapping.  

To this end, we now define the partial temporal NtD measurements on the outer boundary ${\Gamma_\text{1}}$. This is that mapping such that the spatial flux component  $f \in H^{-1/2} (\Gamma_\text{1})$  is mapped to 
$$\widetilde{u}_f( \cdot \, ,t)  \big|_{\Gamma_\text{1}} =  \left\{\begin{array}{cr}
 u_f ( \cdot \, ,t)  \big|_{\Gamma_\text{1}} \, \, & \, \, t \leq T \\ 
 0 \, \, & \, \,  t > T 
 \end{array} \right. \quad \text{ for some } \,\, T \geq 1.$$
It is clear that $\widetilde{u}_f ( \cdot \, ,t)  \big|_{\Gamma_\text{1}}  \in  \text{TD} \left[H^{1/2}(\Gamma_1) \right]$ and denotes the measured partial temporal data on the finite interval time-interval $(0,T)$. This can be seen as an approximation of the measured data where we extend that data for all unknown temporal values by zero. Note that we can write 
\begin{align}\label{partia-data}
\widetilde{u}_f( \cdot \, ,t)  \big|_{\Gamma_\text{1}} = \chi_{_{[0,T]}}(t) \, u_f  ( \cdot \, ,t)  \big|_{\Gamma_\text{1}} \quad \text{for all} \,\,\, t> 0 
\end{align}
 where $\chi$ is the indicator function. 
Now, we will estimate the error in the Laplace Transforms in the NtD measurements with respect to the finite time of measurements taken on $(0,T)$ where $T \geq 1$. 

\begin{theorem}\label{partial-stab}
Let $\widetilde{U}_f (\cdot \, ,s)  \big|_{\Gamma_\text{1}} \in H^{1/2}(\Gamma_1)$ denote the Laplace Transform of the partial temporal NtD measurements given by \eqref{partia-data} for any $f \in H^{-1/2} (\Gamma_\text{1})$. Then we have that there is a $m \in \N$ such that  
$$ \| {U}_f (\cdot \, ,s)-  \widetilde{U}_f (\cdot \, ,s) \|_{H^{1/2}(\Gamma_1)} \leq C T^m \text{e}^{-\text{Re}(s)T}   \|f\|_{H^{-1/2}(\Gamma_1)} \quad \text{ for any } \,\, T\geq 1$$
with the constant $C>0$ is independent of $T$ and $ f \in H^{-1/2} (\Gamma_\text{1})$.
\end{theorem}

\begin{proof}
We begin by noticing that 
$$\big[ u_f  ( \cdot \, ,t) - \widetilde{u}_f( \cdot \, ,t) \big] \big|_{\Gamma_\text{1}} = \big[ 1- \chi_{_{[0,T]}}(t) \big] u_f  ( \cdot \, ,t) \big|_{\Gamma_\text{1}}$$
for any $ f \in H^{-1/2} (\Gamma_\text{1})$. By taking the Laplace Transform on both sides we have that 
$$  {U}_f (\cdot \, ,s)-  \widetilde{U}_f (\cdot \, ,s)  = \int\limits_{T}^{\infty} u ( \cdot \, ,t)  \big|_{\Gamma_\text{1}} \text{e}^{-st} \, \text{d}t.$$ 
From the above equality we are able to estimate the $H^{1/2} (\Gamma_\text{1})$ norm. Therefore, we have that by the Trace Theorem 
\begin{align*}
\| {U}_f (\cdot \, ,s)-\widetilde{U}_f (\cdot \, ,s) \|_{H^{1/2}(\Gamma_1)} &\leq  \int\limits_{T}^{\infty} \| u ( \cdot \, ,t)  \|_{H^{1/2}(\Gamma_1)} \, \text{e}^{-\text{Re}(s)t} \, \text{d}t\\
													  & \leq C \int\limits_{T}^{\infty} \| u ( \cdot \, ,t)  \|_{H^{1}(D_1,\Gamma_0)} \, \text{e}^{-\text{Re}(s)t} \, \text{d}t.
\end{align*}
Now by the norm estimate in Theorem \ref{well-posed} we have that 
$$\| {U}_f (\cdot \, ,s)-\widetilde{U}_f (\cdot \, ,s) \|_{H^{1/2}(\Gamma_1)} \leq  C \|f\|_{H^{-1/2}(\Gamma_1)} \int\limits_{T}^{\infty} t^{\alpha +|1-p|}   \text{e}^{-\text{Re}(s)t} \, \text{d}t$$ 
since we have assumed that $T\geq 1$. We now let $m=\big\lceil{\alpha +|1-p|}\big\rceil$ and whence 
$$  \int\limits_{T}^{\infty} t^{m}   \text{e}^{-\text{Re}(s)t} \, \text{d}t =  \text{e}^{-\text{Re}(s)T} \sum\limits_{k=0}^{m} \binom{m}{k} \frac{(m-k)!}{\text{Re}(s)^{m-k+1}}T^k $$
which is obtained by the Binomial Theorem and using standard Calculus to evaluate the improper integral. Therefore, we can conclude that 
$$\| {U}_f (\cdot \, ,s)-\widetilde{U}_f (\cdot \, ,s) \|_{H^{1/2}(\Gamma_1)} \leq  C \|f\|_{H^{-1/2}(\Gamma_1)} \,  \text{e}^{-\text{Re}(s)T}  \sum\limits_{k=0}^{m} \binom{m}{k} \frac{(m-k)!}{\text{Re}(s)^{m-k+1}}T^k $$
 and again using the fact that $T\geq 1$ proves the claim. 
\end{proof} 
 
Notice that by Theorem \ref{partial-stab} we have that the Laplace Transform of the partial temporal finite time NtD measurements converge in the operator norm to the Laplace Transform of the NtD measurements for \eqref{lt-direct1}$-$\eqref{lt-direct2} as $T \to \infty$. This gives that for the case when  $A=I$ and $c=0$ one can use the stabilized data completion algorithm in \cite{data-completion} to recover the Cauchy data on the inner boundary for some fixed $T \gg 1$ and whence reconstruct the impedance parameters according to \cite{gibc-eit}.

%%%%%%%%%%%%%%%%%%%%%%%%%%%%%%%%%%%%%%%%%%%%%%%%%%%%%%%%%%%%%%%%%%%%%%%%
\section{Summary and Conclusions}
Here we have studied the direct and inverse impedance problems for a sub-diffusion equations with a generalized impedance boundary condition. The analysis for the direct problem holds in both $\R^2$ and $\R^3$. The analysis uses the Laplace Transform to study the problem in the frequency-domain and to assure that one can use the inversion formula to infer the solvability in the time-domain. There is still a need to test numerical methods for solving the direct problem. The uniqueness results for the inverse impedance problem strongly depends on analysis unique to the $\R^2$ case. We have also discussed a possible method for recovering the impedance parameters for the NtD measurements on the outer boundary. The inversion algorithm uses a method for the case when the elliptic operator is given by the Laplacian. A numerical study for the proposed inversion algorithm also needs to be established.

%%%%%%%%%%%%%%%%%%%%%%%%%%%%%%%%%%%%%%%%%%%%%%%%%%%%%%%%%%%%%%%%%


\begin{thebibliography}{99}

\bibitem{data-completion} 
\newblock Y. Boukari and  H. Haddar, 
\newblock A convergent data completion algorithm using surface integral equations, 
\newblock {\it Inverse Problems} {\bf 31} (2015), 035011. 


\bibitem{elastic-laplace} %time-domain for elastic wave eq and fractional time elastic wave eq
\newblock T. Brown, S. Du, H. Eruslu, F.-J. Sayas,
\newblock Analysis of models for viscoelastic wave propagation.
\newblock {\it Applied Mathematics and Nonlinear Sciences}, {\bf 3(1)} (2018) 55--96 


\bibitem{delamination} 
\newblock F. Cakoni, I. De Teresa, H. Haddar and P. Monk,  
\newblock Nondestructive testing of the delaminated interface between two materials, 
\newblock {\it SIAM J. Appl. Math.} {\bf 76} No 6 (2016), 2306-2332.



\bibitem{CK-GIBC} %direct and inverse this problem in electrostatics 
\newblock F. Cakoni and R. Kress, 
\newblock Integral equation methods for the inverse obstacle problem with generalized impedance boundary condition.
\newblock {\it Inverse Problems} {\bf 29} (2013) 015005.


\bibitem{TDLSM}
\newblock F. Cakoni P. Monk and V. Selgas, 
\newblock Analysis of the linear sampling method for imaging penetrable obstacles in the time domain,
\newblock {\it Analysis $\&$ PDEs} (accepted) preprint at: https://sites.math.rutgers.edu/$\sim$fc292/

\bibitem{2nd-order-inclusion}
\newblock S. Chaabane, B. Charfi and H. Haddar, 
\newblock Reconstruction of discontinuous parameters in a second order impedance boundary operator, 
\newblock {\it Inverse Problems} {\bf 32} (2016), 105004.


\bibitem{evans}
\newblock L. Evans, 
\newblock {``Partial Differential Equations''},  
\newblock 2$^{nd}$ edition, AMS 2010.


\bibitem{heat-fm}
\newblock J. Guo, G. Nakamura and H. Wang,
\newblock The factorization method for recovering cavities in a heat conductor
\newblock preprint (2019)  arXiv:1912.11590.




\bibitem{wavelsm}
\newblock H. Haddar, A. Lechleiter, S. Marmorat,
\newblock An improved time domain linear sampling method for Robin and Neumann obstacles,
\newblock {\it Applicable Analysis} , {\bf 93}, issue 2, 2014.


\bibitem{gibc-eit} %sampling method for this problem in electrostatics 
\newblock I. Harris, 
\newblock Detecting an inclusion with a generalized impedance condition from electrostatic data via sampling.
\newblock {\it Mathematical Methods in the Applied Sciences}, DOI: 10.1002/mma.5777 


\bibitem{Harris-Rundell}
\newblock I. Harris and W. Rundell, 
\newblock A direct method for reconstructing inclusions and boundary conditions from electrostatic data, 
\newblock preprint (2017) arXiv:1704.07479.


\bibitem{inv-source-ibc} %1d inv source with impedance bc 
\newblock M. Ismailov and M. Cicek,
\newblock Inverse source problem for a time-fractional diffusion equation with nonlocal boundary conditions.
\newblock {\it Applied Mathematical Modelling}, {\bf 40}, Issues 78, (2016) 4891--4899


\bibitem{frac-tutorial}
\newblock B. Jin and W. Rundell,
\newblock A tutorial on inverse problems for anomalous diffusion processes
\newblock {\it Inverse Problems} {\bf 31} (2015) 035003


\bibitem{inv-potential} %inv parameter problem for frac heat
\newblock B. Kaltenbacher and W. Rundell,
\newblock On an inverse potential problem for a fractional reaction-diffusion equation.
\newblock {\it Inverse Problems} {\bf 35} (2019) 065004



\bibitem{reg-w-frac} %use frac heat for reg method for backward heat
\newblock B. Kaltenbacher and W. Rundell,
\newblock Regularization of a backwards parabolic equation by fractional operators.
\newblock {\it Inverse Problems and Imaging} {\bf 13(2)} (2019) 401--430



\bibitem{kirschbook} 
\newblock A. Kirsch and N. Grinberg,
\newblock  \emph{The Factorization Method for Inverse Problems}. 
\newblock Oxford University Press, Oxford 2008.




\bibitem{numeric-heatlt}
\newblock T. Qiu and et al, 
\newblock Time-domain boundary integral equation modeling of heat transmission problems.
\newblock {\it  Numerische Mathematik} (2019) {\bf 143}:223--259


\bibitem{salsa}
\newblock Sandro Salsa
\newblock ``Partial Differential Equations in Action From Modelling to Theory''
\newblock {\it Springer} 2008.


\bibitem{sayas-book} %time-domain book for laplace transform 
\newblock F.-J. Sayas, 
\newblock { ``Retarded potentials and time domain boundary integral equations: A road map''}
\newblock {\it Springer} 2016.


\bibitem{inv-source-noise} %inv source frac heat with added noise 
\newblock T.-N. Thach et al, 
\newblock Identification of an inverse source problem for time-fractional diffusion equation with random noise
\newblock {\it Mathematical Methods in the Applied Sciences}, {\bf 42} (2019) 204--218

\end{thebibliography}
\end{document}